\documentclass[times,sort&compress,3p]{elsarticle}
\journal{}

\RequirePackage{amsthm,amsmath,amsfonts,amssymb}
\RequirePackage[colorlinks,citecolor=blue,urlcolor=blue]{hyperref}
\RequirePackage{graphicx}
\RequirePackage{bm}
\RequirePackage{color}
\RequirePackage{calrsfs}
\RequirePackage{enumitem}
\setitemize{noitemsep,topsep=0pt,parsep=0pt,partopsep=0pt}   

\makeatletter
\def\ps@pprintTitle{%
 \let\@oddhead\@empty
 \let\@evenhead\@empty
 \def\@oddfoot{}%
 \let\@evenfoot\@oddfoot}
\makeatother

\usepackage[labelfont=bf]{caption}

\numberwithin{equation}{section}

\theoremstyle{plain}
\newtheorem{prop}{Proposition}[section]

\newtheorem{thm}[prop]{Theorem}

\newtheorem{lem}[prop]{Lemma}
\newtheorem{cond}[prop]{Condition}
\theoremstyle{remark}

\newcommand{\eps}{\varepsilon}
\newcommand{\N}{\mathbb{N}}

\newcommand{\R}{\mathbb{R}}

\newcommand{\Cb}{\mathbb{C}}

\newcommand{\Bc}{\mathcal{B}}

\newcommand{\Xc}{\mathcal{X}}

\newcommand{\dd}{\mathrm{d}}

\newcommand{\Ex}{\mathbb{E}}
\newcommand{\Var}{\mathrm{Var}}

\newcommand{\1}{\mathbf{1}}

\renewcommand{\Pr}{\mathbb{P}}

\newcommand{\sss}[1]{\scriptscriptstyle{#1}}
\newcommand{\p}{\overset{\sss{p}}{\to}}
\newcommand{\as}{\overset{\sss a.s.}{\to}}

\allowdisplaybreaks

\begin{document}

\begin{frontmatter}
  
\title{On Stute's representation for a class of smooth, possibly data-adaptive empirical copula processes}

\author{Ivan Kojadinovic\corref{mycorrespondingauthor}}

\address{CNRS / Universit\'e de Pau et des Pays de l'Adour / E2S UPPA, Laboratoire de math\'ematiques et applications -- IPRA, UMR 5142, B.P. 1155, 64013 Pau Cedex, France.}

\cortext[mycorrespondingauthor]{Email address: \url{ivan.kojadinovic@univ-pau.fr}}

\begin{abstract}
Given a random sample from a continuous multivariate distribution, Stute's representation is obtained for empirical copula processes constructed from a broad class of smooth, possibly data-adaptive nonparametric copula estimators. The latter class contains for instance empirical Bernstein copulas introduced by Sancetta and Satchell and thus the empirical beta copula proposed by Segers, Sibuya and Tsukahara. The almost sure rate in Stute's representation is expressed in terms of a parameter controlling the speed at which the spread of the smoothing region decreases as the sample size increases. 
\end{abstract}


\end{frontmatter}


\section{Introduction}

Let $\Xc_n = (\bm X_1,\dots,\bm X_n)$ be a stretch of independent and identically distributed $d$-dimensional random vectors whose unknown distribution function (d.f.) $F$ is assumed to have continuous univariate margins $F_1,\dots,F_d$. From~\cite{Skl59}, the multivariate d.f.\ $F$ can be expressed as
\begin{equation}
\label{eq:sklar}
F(\bm x) = C\{F_1(x_1),\dots,F_d(x_d)\}, \qquad \bm x \in \R^d,
\end{equation}
in terms of a unique \emph{copula} $C$, that is, a unique $d$-dimensional d.f.\ with standard uniform margins which can be thought of as controlling the dependence between the $d$ components of the random vectors in $\Xc_n$. Applications of representation~\eqref{eq:sklar} are numerous: see, e.g., \cite{SalDeMKotRos07,McNFreEmb15,HofKojMaeYan18}.

The best-known nonparametric estimator of $C$ is the \emph{empirical copula} of $\Xc_n$ which we shall define as the empirical d.f.\ of the multivariate ranks obtained from $\Xc_n$ scaled by $1/n$ \cite{Rus76}. Specifically, for any $j \in \{1,\dots,d\}$, let $F_{n,j}$ be the empirical d.f.\ of the $j$th component sample $X_{1j},\dots,X_{nj}$ of $\Xc_n$. Then, $R_{ij} = n F_{n,j}(X_{ij})$ is the rank of $X_{ij}$ among $X_{1j},\dots,X_{nj}$. Next, let $\bm R_i =  (R_{i1}, \dots, R_{id})$, $i \in \{1,\dots,n\}$, be the multivariate ranks obtained from $\Xc_{n}$. The empirical copula $C_n$ of $\Xc_n$ is then defined by
\begin{equation}
\label{eq:Cn}
C_n(\bm u) = \frac{1}{n} \sum_{i=1}^n  \prod_{j=1}^d \1\left( R_{ij}^{1:n} / n \leq u_j \right) 
= \frac{1}{n} \sum_{i=1}^n  \1(\bm R_i / n \leq \bm u), \qquad \bm u \in [0,1]^d,
\end{equation}
where inequalities between vectors are to be understood componentwise. Note that the uniform distance between $C_n$ and the well-known alternative definition due to \cite{Deh79} is smaller than $d/n$, which implies that both definitions are interchangeable in the forthcoming asymptotic results.

The use of $C_n$ to carry out inference on the unknown $C$ in~\eqref{eq:sklar} requires the study of the asymptotics of the empirical copula process $\Cb_n$ defined by $\Cb_n(\bm u) = \sqrt{n} \{C_n(\bm u) - C(\bm u) \}$, $\bm u \in [0,1]^d$; see, e.g., \cite{GanStu87,FerRadWeg04,Tsu05}. The most general results are due to Segers \cite{Seg12} who considered the following non-restrictive condition.

\begin{cond}[Smooth partial derivatives]
\label{cond:pd:smooth}
For any $j \in \{1,\dots,d\}$, the partial derivative $\dot C_j = \partial C/\partial u_j$ exists and is continuous on the set $V_{d,j} = \{ \bm u \in [0, 1]^d : u_j \in (0,1) \}$.
\end{cond}

In the rest of this note, for any $j \in \{1,\dots,d\}$, $\dot C_j$ is arbitrarily defined to be zero on the set $\{ \bm{u} \in [0, 1]^d : u_j \in \{0,1\} \}$, which implies that, under Condition~\ref{cond:pd:smooth}, $\dot C_j$ is defined on the whole of $[0, 1]^d$. Furthermore, let $\bm U_i = (U_{i1},\dots,U_{id})$, $i \in \{1,\dots,n\}$, be the unobservable random vectors obtained by the probability integral transformations $U_{ij} = F_j(X_{ij})$, $j \in \{1,\dots,d\}$. Let $G_n$ be the empirical d.f.\ of $\bm U_1,\dots,\bm U_n$ (which is a random sample from $C$) and let $G_{n,j}$, $j \in \{1,\dots,d\}$, be its $d$ univariate margins. The corresponding empirical processes, $\alpha_n$ and $\alpha_{n,j}$, are respectively defined by $\alpha_n(\bm u) = \sqrt{n} \{G_n(\bm u) - C(\bm u)\}$ and $\alpha_{n,j}(u_j) = \sqrt{n} \{G_{n,j}(u_j) - u_j\}$, $\bm u \in [0,1]^d$. As we continue, convergences are as $n \to \infty$.

One of the main results obtained in \cite{Seg12} (see Proposition 3.1 therein) is that, under Condition~\ref{cond:pd:smooth}, 
\begin{equation}
  \label{eq:tilde:Cbn}
\sup_{\bm u \in [0,1]^d} | \Cb_n(\bm u) -\tilde \Cb_n(\bm u)  | \p 0, \qquad \text{where} \qquad  \tilde \Cb_n(\bm u) = \alpha_n(\bm u) - \sum_{j=1}^d \dot C_j(\bm u) \alpha_{n,j}(u_j), \qquad \bm u \in [0,1]^d,
\end{equation}
which implies weak convergence of the empirical copula process $\Cb_n$ to the usual well-identified limit in the literature; see, e.g., \cite{GanStu87,FerRadWeg04,Tsu05,Seg12}. The convergence result in~\eqref{eq:tilde:Cbn} is also instrumental for deriving and asymptotically validating resampling schemes for approximating the ``sampling distribution'' of $C_n$ in~\eqref{eq:Cn}; see, e.g., \cite{RemSca09,Seg12,KojSte19}.

Stute's representation of $\Cb_n$, conjectured in Section~4 of~\cite{Stu84}, is a strengthening of the convergence in probability in~\eqref{eq:tilde:Cbn}. Its proof was given in \cite{Seg12} under certain growth conditions on the second-order partial derivatives of $C$ that allow for explosive behavior near the boundaries. 

\begin{cond}[Smooth second-order partial derivatives]
  \label{cond:so:pd}
  For any $i,j \in \{1,\dots,d\}$, the second-order partial derivative $\ddot C_{ij} = \partial^2 C/(\partial u_i \partial u_j)$ exists and is continuous on the set $V_{d,i} \cap V_{d,j}$, and there exists a constant $L > 0$ such that
$$
|\ddot C_{ij}(\bm u)| \leq L \min \left( \frac{1}{u_i(1-u_i)}, \frac{1}{u_j(1-u_j)} \right), \qquad \bm u \in V_{d,i} \cap V_{d,j}.
$$
\end{cond}

Proposition 4.2 of \cite{Seg12} then states that, under Conditions~\ref{cond:pd:smooth} and~\ref{cond:so:pd},
\begin{equation}
  \label{eq:asr:Cbn}
  \sup_{\bm u \in [0,1]^d} | \Cb_n(\bm u) -\tilde \Cb_n(\bm u)  | = O(n^{-1/4} (\log n)^{1/2} (\log \log n)^{1/4}) \qquad \text{almost surely}.  
\end{equation}
Several applications of~\eqref{eq:asr:Cbn} are discussed in Section~4 of~\cite{Stu84}. Additional applications concern \emph{open-end sequential change-point detection};  see, e.g., \cite{KirWeb18,GosKleDet21,HolKoj21}. For instance, one way of establishing the asymptotics of certain procedures of this type based on $C_n$ in~\eqref{eq:Cn} (for monitoring changes in the copula) requires among other things to prove that $\sup_{m > n} Y_m \p 0$, where $Y_n = n^{-1/2} \max_{1 \leq k \leq n} \sup_{\bm u \in [0,1]^d} k^{1/2} |\Cb_k(\bm u) - \tilde \Cb_k(\bm u)|$. Showing the latter is highly none-trivial in general, but if Stute's representation for $\Cb_n$ holds, it follows from the fact that $Y_n \as 0$ as a consequence of~\eqref{eq:asr:Cbn}.

The aim of this note is to obtain Stute's representation for the smooth, possibly data-adaptive empirical copula processes recently considered in \cite{KojYi22}. The latter processes are based on smooth nonparametric estimators of $C$ that can be substantially better-behaved than $C_n$ in~\eqref{eq:Cn} in finite samples. Specifically, assume that, for any $n \in \N$:
  \begin{itemize}
  \item  for any $\bm x \in (\R^d)^n$ and $\bm u \in [0,1]^d$, $\nu_{\bm u}^{\bm x}$ is the law of a $[0,1]^d$-valued mean $\bm u$ random vector $\bm W_{\bm u}^{\bm x}$ the components of which are denoted by $W_{1,u_1}^{\bm x}, \dots, W_{d,u_d}^{\bm x}$ to indicate that the $j$th component depends on $u_j$ but not on $u_1,\dots,u_{j-1},u_{j+1},\dots,u_d$,
  \item for any $\bm u \in [0,1]^d$ and $S \in \Bc([0,1]^d)$, $\bm x \mapsto \nu_{\bm u}^{\bm x}(S)$ is a measurable function from $\big( (\R^d)^n,  \Bc((\R^d)^n) \big)$ to $\big( [0,1], \Bc([0,1])\big)$.
  \end{itemize}
A broad class of smooth versions of $C_n$ in~\eqref{eq:Cn}, with possibly data-adaptive smoothing, is then given by
\begin{equation}
  \label{eq:Cn:nu}
  C_n^\nu(\bm u) = \int_{[0,1]^d} C_n(\bm w) \dd \nu_{\bm u}^{\sss \Xc_n}(\bm w), \qquad \bm u \in [0,1]^d.
\end{equation}
Roughly speaking, for any $\bm u \in [0,1]^d$, $C_n^\nu(\bm u)$ can be thought of as a ``weighted average'' of $C_n(\bm w)$ for $\bm w$ ``in a neighborhood of $\bm u$'' according to the smoothing distribution $\nu_{\bm u}^{\sss \Xc_n}$ (that may depend on the observations $\Xc_n$). \emph{Empirical Bernstein copulas} introduced in \cite{SanSat04} and the \emph{empirical beta copula} proposed in \cite{SegSibTsu17} belong the above-defined class; see \cite{KojYi22} for more details as well as \cite{SegSibTsu17} on which the aforementioned reference heavily relies.

\section{Main result}

The smooth empirical copula process $\Cb_n^\nu$ corresponding to $C_n^\nu$ in~\eqref{eq:Cn:nu} is defined by $\Cb_n^\nu(\bm u) = \sqrt{n} \{ C_n^\nu(\bm u) - C(\bm u) \}$, $\bm u \in [0,1]^d$. Building upon \cite{SegSibTsu17}, its asymptotics were obtained in \cite{KojYi22} under a condition similar to the following one.

\begin{cond}[Variance condition]
  \label{cond:var:W}
  There exists constants $\kappa > 0$ and $\gamma \in [1,2)$ such that, for any $n \in \N$, $\bm x \in (\R^d)^n$, $\bm u \in [0,1]^d$ and $j \in \{1,\dots,d\}$, $\Var( W_{j,u_j}^{\bm x}) \leq \kappa u_j(1-u_j) / n^\gamma$.\end{cond}

Clearly, if Condition~\ref{cond:var:W} holds with $\gamma = \gamma_0 \in (1,2)$, it holds for any $\gamma \in [1,\gamma_0)$. Theorem~6.4 in \cite{KojYi22} implies that, if Condition~\ref{cond:pd:smooth} and Condition~\ref{cond:var:W} with $\gamma = 1$ hold, $\sup_{\bm u \in [0,1]^d} | \Cb_n^\nu(\bm u) -\tilde \Cb_n(\bm u)  | \p 0$, where $\tilde \Cb_n$ is defined in~\eqref{eq:tilde:Cbn}.  Compared to \cite{KojYi22}, we consider the above more precise formulation of the variance condition in order to express the almost sure rate in Stute's representation for $\Cb_n^\nu$ in terms of the rate at which the spread of the smoothing distributions decreases.

\begin{thm}
  \label{thm:asr}
  Under Conditions~\ref{cond:pd:smooth},~\ref{cond:so:pd} and~\ref{cond:var:W},
\begin{equation}
  \label{eq:asr}
  \sup_{\bm u \in [0,1]^d} | \Cb_n^\nu(\bm u) - \tilde \Cb_n(\bm u)| = O(n^{(3 - 4\gamma)/6}) + O(n^{-1/4} (\log n)^{1/2} (\log \log n)^{1/4}) + O(n^{-\gamma/5} (\log n)^{1/2} (\log \log n)^{1/2})\quad \text{almost surely},
\end{equation}
where $\gamma$ is the constant appearing in Condition~\ref{cond:var:W}.
\end{thm}

As expected, the faster the spread of the smoothing distributions decreases, the better the order of the approximation: from~\eqref{eq:asr}, we obtain that the order is $O(n^{-1/6})$ if Condition~\ref{cond:var:W} holds with $\gamma = 1$, while it becomes equal to the one in~\eqref{eq:asr:Cbn} for the usual empirical copula process if Condition~\ref{cond:var:W} holds with $\gamma > 5/4$.

In future work, Theorem~\ref{thm:asr} will be used to establish the asymptotics of certain open-end sequential change-point detection procedures based on $C_n^\nu$ in~\eqref{eq:Cn:nu} for monitoring changes in the copula of multivariate obseravtions.

\section{Proof of Theorem~\ref{thm:asr}}

Let $\tilde \Cb_n^\nu(\bm u) = \int_{[0,1]^d} \tilde \Cb_n(\bm w)\dd \nu_{\bm u}^{\sss \Xc_n}(\bm w)$, $\bm u \in [0,1]^d$, where $\tilde \Cb_n$ is defined in \eqref{eq:tilde:Cbn}. Then, we have the decomposition
\begin{multline}
  \label{eq:decomp}
\sup_{\bm u \in [0,1]^d} | \Cb_n^\nu(\bm u) - \tilde \Cb_n(\bm u) | \leq \sup_{\bm u \in [0,1]^d} \sqrt{n} \left| \int_{[0,1]^d} C(\bm w)\dd \nu_{\bm u}^{\sss \Xc_n}(\bm w) - C(\bm u) \right| \\ + \sup_{\bm u \in [0,1]^d} \left|\int_{[0,1]^d} \{\Cb_n(\bm w) - \tilde \Cb_n(\bm w) \} \dd \nu_{\bm u}^{\sss \Xc_n}(\bm w) \right| + \sup_{\bm u \in [0,1]^d} | \tilde \Cb_n^\nu(\bm u) - \tilde \Cb_n(\bm u) |.
\end{multline}
The second supremum on the right-hand side of~\eqref{eq:decomp} is smaller $\sup_{\bm u \in [0,1]^d} | \Cb_n(\bm u) - \tilde \Cb_n(\bm u)|$ which, according to~\eqref{eq:asr:Cbn}, is $O(n^{-1/4} (\log n)^{1/2} (\log \log n)^{1/4})$ almost surely. The claim then follows from Lemmas~\ref{lem:bias} and~\ref{lem:stoc} below. The former can be regarded as an extension of Proposition 3.5 in \cite{SegSibTsu17}. More generally, many of the arguments used in the proofs of the lemmas are adapted from arguments used in \cite{Seg12,SegSibTsu17}.

\begin{lem}
\label{lem:bias}
Under Conditions~\ref{cond:pd:smooth},~\ref{cond:so:pd} and~\ref{cond:var:W},
\begin{equation}
  \label{eq:bias}
  \sup_{\bm u \in [0, 1]^d} \sqrt{n} \left| \int_{[0,1]^d} C(\bm w) \dd \nu_{\bm u}^{\sss \Xc_n}(\bm w)-C(\bm u) \right| = O(n^{(3 - 4\gamma)/6}), \qquad \text{almost surely}.
\end{equation}
\end{lem}

\begin{proof}[\bf Proof]
To prove the claim, it suffices to show that it holds conditionally on $\bm X_1,\bm X_2,\dots$ for almost any sequence $\bm X_1,\bm X_2,\dots$. We thus reason conditionally on $\bm X_1,\bm X_2,\dots$ in the rest of this proof.  

Let $\bm u, \bm w \in [0,1]^d$, and define $\bm w(t) = \bm u + t (\bm w - \bm u)$, $t \in [0,1]$ and $G(t) = C\{\bm w(t) \}$, $t \in [0,1]$. The function $G$ is continuous on $[0,1]$ and is continuously differentiable on $(0,1)$ by Condition~\ref{cond:pd:smooth} with derivative $G'(t) = \sum_{j=1}^d (w_j - u_j) \dot C_j \{ \bm w(t)\}$, $t \in (0,1)$. By the fundamental theorem of calculus, $G(1) - G(0) = \int_0^1 G'(t) \dd t$, that is,
$$
C(\bm w) - C(\bm u) = \sum_{j=1}^d (w_j - u_j)  \int_0^1 \dot C_j \{ \bm w(t)\} \dd t.
$$
Some thought reveals that, under Condition~\ref{cond:pd:smooth} and with the adopted conventions, the previous equality holds no matter how $\bm u$ and $\bm w$ are chosen in $[0,1]^d$. Using Fubini's theorem, the left-hand side of~\eqref{eq:bias} is then equal to 
$$
\sup_{\bm u \in [0, 1]^d} \sqrt{n} \left| \int_{[0,1]^d} \{ C(\bm w) - C(\bm u) \} \dd \nu_{\bm u}^{\sss \Xc_n}(\bm w) \right| = \sup_{\bm u \in [0, 1]^d} \sqrt{n} \left|\sum_{j=1}^d\int_0^1\left\{ \int_{[0,1]^d}(w_j-u_j)\dot{C}_j \{ \bm w(t)\} \dd \nu_{\bm u}^{\sss \Xc_n}(\bm w)\right\} \dd t \right|.
$$
For any $j \in \{1,\dots,d\}$, let
\begin{equation}
  \label{eq:Ijn}
  I_{j,n} = \sqrt{n} \int_0^1\sup_{\bm u \in [0, 1]^d}\left|\int_{[0,1]^d}(w_j-u_j)\dot{C}_j \{ \bm w(t) \} \dd \nu_{\bm u}^{\sss \Xc_n}(\bm w)\right|\dd t.
\end{equation}
By the triangle inequality, the left-hand side of~\eqref{eq:bias} is then smaller than $\sum_{j=1}^d I_{j,n}$. Fix $j \in \{1,\dots,d\}$. To prove~\eqref{eq:bias}, we shall now show that $I_{j,n} = O(n^{(3 - 4\gamma)/6})$. Let $\delta_n = n^{-\xi}$ for some $\xi \in (0,\gamma - 1/2)$ to be determined later. For $n$ sufficiently large such that $\delta_n \leq 1/2$, we have that $I_{j,n} \leq J_{j,n} + K_{j,n}$, where
\begin{align}
  \nonumber
  J_{j,n} &= \sqrt{n} \int_0^1\sup_{\substack{\bm u \in [0, 1]^d \\ u_j \in [0,\delta_n)\cup(1-\delta_n,1]}} \left| \int_{[0,1]^d}(w_j-u_j) \dot{C}_j\{\bm w(t)\} \dd \nu_{\bm u}^{\sss \Xc_n}(\bm w) \right| \dd t, \\
  \label{eq:K:j:n:delta}
  K_{j,n} &= \sqrt{n} \int_0^1\sup_{\substack{\bm u \in [0, 1]^d \\ u_j \in [\delta_n,1-\delta_n]}} \left| \int_{[0,1]^d}(w_j-u_j) \dot{C}_j\{\bm w(t)\} \dd \nu_{\bm u}^{\sss \Xc_n}(\bm w) \right| \dd t.
\end{align}

\emph{Term $J_{j,n}$:} Since $0 \leq \dot C_j \leq 1$ \cite[see, e.g.,][Section 2.2]{Nel06}, and from H\"older's inequality and Condition~\ref{cond:var:W}, 
\begin{align*}
  J_{j,n} \leq& \sqrt{n} \sup_{\substack{\bm u \in [0, 1]^d \\ u_j \in [0,\delta_n)\cup(1-\delta_n,1]}}\int_{[0,1]^d}|w_j-u_j|\dd \nu_{\bm u}^{\sss \Xc_n}(\bm w) \leq  \sqrt{n} \sup_{\substack{\bm u \in [0, 1]^d \\ u_j \in [0,\delta_n)\cup(1-\delta_n,1]}} \sqrt{ \int_{[0,1]^d} \{ w_j-\Ex(W_{j,u_j}^{\sss \Xc_n} \mid \bm X_1, \bm X_2, \dots) \}^2\dd \nu_{\bm u}^{\sss \Xc_n}(\bm w)} \\
  =&  \sqrt{n} \sup_{\substack{\bm u \in [0, 1]^d \\ u_j \in [0,\delta_n)\cup(1-\delta_n,1]}} \sqrt{\Var(W_{j,u_j}^{\sss \Xc_n} \mid \bm X_1, \bm X_2, \dots)} \leq n^{(1-\gamma)/2} \sup_{u \in [0,\delta_n)\cup(1-\delta_n,1]}\sqrt{\kappa u(1-u)} = O( n^{(1-\gamma)/2} \delta_n^{1/2}) = O(n^{(1 - \gamma -\xi)/2}),
\end{align*}
since $0 \leq u \leq \delta_n$ implies that $\sqrt{u(1-u)} \leq \sqrt{u} \leq \delta_n^{1/2}$ and $1 - \delta_n \leq u \leq 1$ implies that $\sqrt{u(1-u)} \leq \sqrt{1-u} \leq \delta_n^{1/2}$.

\emph{Term $K_{j,n}$:} Since $\int_{[0,1]^d}(w_j - u_j)\dd \nu_{\bm u}^{\sss \Xc_n}(\bm w) = 0$, $K_{j,n}$ in~\eqref{eq:K:j:n:delta} can be rewritten as
$$
 K_{j,n} = \sqrt{n} \int_0^1\sup_{\substack{\bm u \in [0, 1]^d \\ u_j \in [\delta_n,1-\delta_n]}} \left| \int_{[0,1]^d}(w_j-u_j) \left[ \dot{C}_j\{\bm w(t)\} -\dot C_j(\bm u) \right] \dd \nu_{\bm u}^{\sss \Xc_n}(\bm w) \right| \dd t.
$$
Let $\eps_n = \delta_n/2$ for all $n \in \N$. Then, $K_{j,n} \leq K_{j,n}' + K_{j,n}''$, where
\begin{align*}
  K_{j,n}' =& \sqrt{n} \int_0^1\left[\sup_{\substack{\bm u \in [0, 1]^d \\ u_j \in [\delta_n,1-\delta_n]}} \int_{\{\bm w \in [0,1]^d : |\bm w-\bm u|_{\infty}\leq \eps_n \}} |w_j-u_j| \left| \dot{C}_j \{ \bm w(t) \} - \dot{C}_j(\bm u) \right| \dd \nu_{\bm u}^{\sss \Xc_n}(\bm w)\right]\dd t, \\
  K_{j,n}'' =& \sqrt{n} \int_0^1 \left[\sup_{\substack{\bm u \in [0, 1]^d \\ u_j \in [\delta_n,1-\delta_n]}} \int_{[0,1]^d} \1\{|\bm w-\bm u|_\infty>\eps_n\} |w_j-u_j| \left| \dot{C}_j \{ \bm w(t) \} - \dot{C}_j(\bm u) \right| \dd \nu_{\bm u}^{\sss \Xc_n}(\bm w) \right]\dd t.
\end{align*}
From Condition~\ref{cond:so:pd} and Lemma 4.3 in \cite{Seg12}, for all $\bm u, \bm v \in [0,1]^d$ such that $u_j,v_j \in (0,1)$,
\begin{equation}
  \label{eq:so:pd}
|\dot C_j(\bm u) - \dot C_j(\bm v)| \leq L \max \left\{ \frac{1}{u_j(1-u_j)}, \frac{1}{v_j(1-v_j)} \right\} \sum_{k=1}^d |u_k - v_k|,
\end{equation}
which implies that, for any $\bm u, \bm w \in [0,1]^d$ such that $u_j \in [\delta_n,1-\delta_n]$ and $|\bm w-\bm u|_{\infty}\leq \eps_n$, and for any $t \in[0,1]$,
\begin{align*}
  |\dot C_j\{\bm w(t)\} - \dot C_j(\bm u)| &\leq t L \max \left[ \frac{1}{u_j(1-u_j)}, \frac{1}{\{u_j + t(w_j- u_j)\}\{1 - u_j - t(w_j- u_j)\}} \right] \sum_{k=1}^d |w_k - u_k| \leq L L' \delta_n^{-1} \sum_{k=1}^d |w_k - u_k|,
\end{align*}
where $L' > 0$ is another constant. Indeed, $u_j \in [\delta_n, 1- \delta_n]$ which implies that $u_j(1-u_j) \geq \delta_n(1-\delta_n)$ and thus that $\{u_j(1-u_j) \}^{-1}\leq \{\delta_n(1-\delta_n)\}^{-1} = O(\delta_n^{-1})$. Similarly, $u_j + t(w_j- u_j) \in [\delta_n - \eps_n, 1- \delta_n+\eps_n] = [\delta_n/2, 1-\delta_n/2]$ which implies that $\{u_j + t(w_j- u_j)\}\{1 - u_j - t(w_j- u_j)\}^{-1} = O(\delta_n^{-1})$. Hence,
\begin{align*}
  K_{j,n}' \leq& LL' \sqrt{n} \delta_n^{-1} \sup_{\bm u \in [0, 1]^d} \int_{[0,1]^d} |w_j-u_j| \sum_{k=1}^d |w_k - u_k|\dd \nu_{\bm u}^{\sss \Xc_n}(\bm w) \leq LL' \sqrt{n} \delta_n^{-1}  \sum_{k=1}^d \sup_{\bm u \in [0, 1]^d} \int_{[0,1]^d} |w_j-u_j| |w_k - u_k|\dd \nu_{\bm u}^{\sss \Xc_n}(\bm w).
\end{align*}
From H\"older's inequality and Condition~\ref{cond:var:W}, for any $\bm u \in [0, 1]^d$ and $k \in \{1,\dots,d\}$,
\begin{align*}
  \int_{[0,1]^d} |w_j-u_j| |w_k - u_k|\dd \nu_{\bm u}^{\sss \Xc_n}(\bm w) &= \Ex(|W_{j,u_j}^{\sss \Xc_n} - u_j||W_{k,u_k}^{\sss \Xc_n} - u_k| \mid \bm X_1, \bm X_2, \dots) \\
                                                                          &\leq \sqrt{\Var(W_{j,u_j}^{\sss \Xc_n} \mid \bm X_1, \bm X_2, \dots)} \sqrt{\Var(W_{k,u_k}^{\sss \Xc_n} \mid \bm X_1, \bm X_2, \dots)} \leq \kappa n^{-\gamma}.
\end{align*}
It thus follows that $K_{j,n}' = O(n^{1/2} \delta_n^{-1} n^{-\gamma}) = O(n^{1/2 + \xi - \gamma}) $. As far as $K_{j,n}''$ is concerned, using the fact that $0 \leq \dot C_j \leq 1$ and that
$$
\1\{|\bm w-\bm u|_\infty>\eps_n\} \leq \frac{1}{\eps_n} |\bm w-\bm u|_\infty \leq \frac{1}{\eps_n} \sum_{k=1}^d |w_k - u_k|,
$$
we obtain that
\begin{align}
  \nonumber
  K_{j,n}'' &\leq \sqrt{n} \eps_n^{-1} \sum_{k=1}^d \sup_{\bm u \in [0, 1]^d} \int_{[0,1]^d}|w_j-u_j| |w_k - u_k| \dd \nu_{\bm u}^{\sss \Xc_n}(\bm w) = O(n^{1/2} \eps_n^{-1} n^{-\gamma}) = O(n^{1/2} \delta_n^{-1} n^{-\gamma}) = O(n^{1/2 + \xi - \gamma})
\end{align}
since $\eps_n = \delta_n/2$. Thus, $I_{j,n}$ in~\eqref{eq:Ijn} is $O(n^{(1 - \gamma -\xi)/2}) + O(n^{1/2 + \xi - \gamma})$. Some thought reveals that the best rate is obtained by taking $\xi = \gamma/3$ which gives $I_{j,n} = O(n^{(3 - 4\gamma)/6})$. The latter holds conditionally on $\bm X_1, \bm X_2, \dots$ for almost any sequence $\bm X_1, \bm X_2, \dots$, which completes the proof of~\eqref{eq:bias}.
\end{proof}





\begin{lem}
\label{lem:stoc}
Under Conditions~\ref{cond:pd:smooth},~\ref{cond:so:pd} and~\ref{cond:var:W},
\begin{equation}
  \label{eq:stoc}
  \sup_{\bm u \in [0,1]^d} | \tilde \Cb_n^\nu(\bm u) - \tilde \Cb_n(\bm u) | = O(n^{-\gamma/5} (\log n)^{1/2} (\log \log n)^{1/2}) \qquad \text{almost surely}. 
\end{equation}
\end{lem}

\begin{proof}
Let $\eps_n = n^{-\xi}$ for some $\xi \in (0,\gamma/2)$ to be specified later. The left-hand side of~\eqref{eq:stoc} can then be decomposed, for $n$ sufficiently large, as 
\begin{equation*}
  \sup_{\bm u \in [0,1]^d} | \tilde \Cb_n^\nu(\bm u) - \tilde \Cb_n(\bm u) | = \sup_{\bm u \in [0,1]^d} \left| \int_{[0,1]^d} \{\tilde \Cb_n(\bm w) - \tilde \Cb_n(\bm u) \} \dd \nu_{\bm u}^{\sss \Xc_n}(\bm w) \right| \leq Q_n + R_n,
\end{equation*}
where 
\begin{align}
  \nonumber
  Q_n &= \sup_{\bm u \in [0,1]^d}  \left| \int_{\substack{\bm w \in [0,1]^d \\ | \bm w - \bm u |_\infty \geq \eps_n}} \{\tilde \Cb_n(\bm w) - \tilde \Cb_n(\bm u) \} \dd \nu_{\bm u}^{\sss \Xc_n}(\bm w) \right|, \\
  \label{eq:Rn}
  R_n &= \sup_{\bm u \in [0,1]^d} \left| \int_{\substack{\bm w \in [0,1]^d \\ | \bm w - \bm u |_\infty \leq \eps_n}} \{\tilde \Cb_n(\bm w) - \tilde \Cb_n(\bm u) \} \dd \nu_{\bm u}^{\sss \Xc_n}(\bm w) \right|.
\end{align}
Using the expression of $\tilde \Cb_n$ in~\eqref{eq:tilde:Cbn} and the fact that $0 \leq \dot C_j \leq 1$ for all $j \in \{1,\dots,d\}$, we have that
\begin{align*}
  Q_n \leq 2 (d+1) \sup_{\bm u \in [0,1]^d} |\alpha_n(\bm u)| \times \sup_{\bm u \in [0,1]^d} \nu_{\bm u}^{\sss \Xc_n}(\{ \bm w \in [0,1]^d : | \bm w - \bm u |_\infty \geq \eps_n \}).
\end{align*}
From the law of the iterated logarithm for empirical processes \citep[see, e.g.,][Chap. 2]{Kos08}, we have that $\sup_{\bm u \in [0,1]^d} |\alpha_n(\bm u)| = O((\log \log n)^{1/2})$ almost surely. Furthermore, for almost any sequence $\bm X_1,\bm X_2,\dots$, conditionally on $\bm X_1,\bm X_2,\dots$, using Chebyshev's inequality and Condition~\ref{cond:var:W}, we obtain that
\begin{align*}
  \nu_{\bm u}^{\sss \Xc_n}(\{ \bm w \in [0,1]^d : | \bm w - \bm u |_\infty \geq \eps_n \}) &= \Pr( | \bm W_{\bm u}^{\sss \Xc_n} - \bm u |_\infty \geq \eps_n \mid \bm X_1,\bm X_2,\dots) = \Pr \left[ \bigcup_{j=1}^d \left\{ | W_{j,u_j}^{\sss \Xc_n} - u_j |_\infty \geq \eps_n \right\} \mid \bm X_1,\bm X_2,\dots \right] \\
  &\leq \sum_{j=1}^d \Pr \left( | W_{j,u_j}^{\sss \Xc_n} - u_j |_\infty \geq \eps_n \mid \bm X_1,\bm X_2,\dots \right) \leq \sum_{j=1}^d \frac{\Var(W_{j,u_j}^{\sss \Xc_n} \mid \bm X_1,\bm X_2,\dots)}{\eps_n^2} \leq \frac{\kappa d}{n^\gamma \eps_n^2}.
\end{align*}
It follows that $Q_n = O(n^{-\gamma} \eps_n^{-2} (\log \log n)^{1/2}) = O(n^{-\gamma + 2\xi} (\log \log n)^{1/2})$ almost surely.

We now deal with the term $R_n$ in~\eqref{eq:Rn}. Recall for instance from~\cite{Stu84} that the oscillation modulus of the multivariate empirical process $\alpha_n$ is defined by
$$
M_n(\eps) = \sup_{\substack{\bm v, \bm w \in [0,1]^d \\ | \bm u - \bm w |_\infty \leq \eps}} | \alpha_n(\bm u) - \alpha_n(\bm w)|,  \qquad \eps \in [0,1].
$$
Furthermore, for any $j \in \{1,\dots,d\}$, let
\begin{equation}
  \label{eq:Rnj}
R_{n,j} = \sup_{\bm u \in [0,1]^d} \int_{\substack{\bm w \in [0,1]^d \\ | \bm w - \bm u |_\infty \leq \eps_n}} |\dot C_j(\bm u) - \dot C_j(\bm w)| |\alpha_{n,j}(u_j)| \dd \nu_{\bm u}^{\sss \Xc_n}(\bm w).
\end{equation}
Then, using the fact that $0 \leq \dot C_j \leq 1$ for all $j \in \{1,\dots,d\}$, it can be verified from~\eqref{eq:Rn} that
\begin{equation}
  \label{eq:Rn:ineq}
R_n \leq (d+1) M_n(\eps_n) + \sum_{j=1}^d  R_{n,j}.
\end{equation}
For the first summand on the right-end side of~\eqref{eq:Rn:ineq}, we proceed as in the proof of Proposition 4.2 of \cite{Seg12} for the term $I_n$. Let $\lambda_n = 2 K_2^{-1/2} (\log n)^{1/2} \eps_n^{1/2}$, where $K_2$ is the constant in Proposition A.1 of \cite{Seg12}. Since $\lambda_n / (n^{1/2} \eps_n) = O(n^{(\xi-1)/2}(\log n)^{1/2}) \to 0$ as $\xi < \gamma/2 < 1$ and the function $\psi$ in (A.2) of \cite{Seg12} is decreasing with $\psi(0) = 1$, we have that $\psi \{ \lambda_n / (n^{1/2} \eps_n)\} \geq 1/2$ for all $n$ greater than some $n_0$. Hence, 
$$
\sum_{n=n_0}^\infty \frac{1}{\eps_n} \exp\left\{- \frac{K_2 \lambda_n^2}{\eps_n} \psi \left( \frac{\lambda_n}{\sqrt{n}\eps_n} \right) \right\} \leq \sum_{n=n_0}^\infty \frac{1}{\eps_n} \exp\left( - \frac{K_2 \lambda_n^2}{2\eps_n}  \right) = \sum_{n=n_0}^\infty \frac{1}{\eps_n} \exp\left( - 2\log n  \right) =  \sum_{n=n_0}^\infty \frac{1}{n^2 \eps_n} =  \sum_{n=n_0}^\infty \frac{1}{n^{2-\xi}} < \infty. 
$$
Therefore, by the Borel--Cantelli lemma, $M_n(\eps_n) = O(\lambda_n) = O(\eps_n^{1/2} (\log n)^{1/2}) = O(n^{-\xi/2} (\log n)^{1/2})$ almost surely.

Fix $j \in \{1,\dots,d\}$ and let us deal with $R_{n,j}$ in~\eqref{eq:Rnj}. Let $\delta_n = \eps_n \log n = n^{-\xi} \log n$. It can be verified that
\begin{equation}
  \label{eq:Rnj:ineq}
R_{n,j} \leq \sup_{u_j \in [0,\delta_n)\cup(1-\delta_n,1]} |\alpha_{n,j}(u_j)| +  \sup_{\substack{\bm u \in [0,1]^d \\ u_j \in [\delta_n, 1-\delta_n]}} \int_{\substack{\bm w \in [0,1]^d \\ | \bm w - \bm u |_\infty \leq \eps_n}} |\dot C_j(\bm u) - \dot C_j(\bm w)| |\alpha_{n,j}(u_j)| \dd \nu_{\bm u}^{\sss \Xc_n}(\bm w).
\end{equation}
Proceeding as in the proof of Proposition~4.2 of \cite{Seg12}, from Theorem 2~(iii) in \cite{EinMas88}, the first supremum on the right-hand side of the previous display is $O(\delta_n^{1/2} (\log \log n)^{1/2}) = O(n^{-\xi/2} (\log n)^{1/2} (\log \log n)^{1/2})$ almost surely. As far as the second supremum is concerned, note that, for any $\bm u, \bm w \in [0,1]^d$ such that $u_j \in [\delta_n,1-\delta_n]$ and $|\bm w-\bm u|_{\infty}\leq \eps_n$, we have
$$
w_j = u_j \left(1 + \frac{w_j - u_j}{u_j} \right) \geq  u_j \left(1 - \frac{|w_j - u_j|}{u_j} \right) \geq  u_j \left(1 - \frac{\eps_n}{\delta_n} \right) = u_j \left(\frac{\log n - 1}{\log n} \right) , 
$$
and similarly that
$$
1 - w_j = (1 - u_j) \left(1 + \frac{u_j - w_j}{1 - u_j} \right) \geq  (1 - u_j) \left(1 - \frac{|w_j - u_j|}{1 - u_j} \right) \geq  (1 - u_j) \left(1 - \frac{\eps_n}{\delta_n} \right) = (1 - u_j) \left(\frac{\log n - 1}{\log n} \right)  
$$
so that $w_j(1-w_j) \geq u_j(1-u_j) (\log n - 1)^2 / (\log n)^2$ and thus
$$
\max \left\{ \frac{1}{u_j(1-u_j)}, \frac{1}{w_j(1-w_j)} \right\} \leq \left(\frac{\log n}{\log n - 1}\right)^{2} \frac{1}{u_j(1-u_j)}.
$$
Combined with~\eqref{eq:so:pd} which holds from Condition~\ref{cond:so:pd} and Lemma 4.3 in \cite{Seg12}, it follows that, for sufficiently large $n$, the second supremum on the right-hand side of~\eqref{eq:Rnj:ineq} is smaller than
$$
2L \sup_{\substack{\bm u \in [0,1]^d \\ u_j \in [\delta_n, 1-\delta_n]}} \int_{\substack{\bm w \in [0,1]^d \\ | \bm w - \bm u |_\infty \leq \eps_n}} \frac{|\alpha_{n,j}(u_j)|}{u_j(1-u_j)} \sum_{k=1}^d |u_k - w_k| \dd \nu_{\bm u}^{\sss \Xc_n}(\bm w) \leq 2 L \sup_{u_j \in [\delta_n,1-\delta_n]} \frac{|\alpha_{n,j}(u_j)|}{u_j(1-u_j)} \times \sum_{k=1}^d \sup_{\bm u \in [0,1]^d} \int_{[0,1]^d} |u_k - w_k| \dd \nu_{\bm u}^{\sss \Xc_n}(\bm w).
$$
The second factor is $O(n^{-\gamma/2})$ almost surely by H\"older's inequality and Condition~\ref{cond:var:W} (see the treatment of the term $J_{n,j}$ in the proof of Lemma~\ref{lem:bias}) while the first factor is, with probability one, $O(\delta_n^{-1/2} b_n)$ with $b_n = (\log n)^{1/2} \log \log n$ since, from \cite{Mas81} (see also the proof of Proposition 4.2 in \cite{Seg12}), with probability one, for all $n$ sufficiently large, $|\alpha_{n,j}(u_j)| \leq \{ u_j(1-u_j) \}^{1/2} b_n$ for all $u_j \in [0,1]$. Hence, the second term on the right-hand side of~\eqref{eq:Rnj:ineq} is $O(\delta_n^{-1/2} b_n n^{-\gamma/2}) = O(n^{(\xi-\gamma)/2} \log \log n)$ almost surely, which implies that $R_{n,j} = O(n^{-\xi/2} (\log n)^{1/2} (\log \log n)^{1/2}) +  O(n^{(\xi-\gamma)/2} \log \log n)$ almost surely and thus that
\begin{align*}
 Q_n + R_n &= O(n^{-\gamma + 2\xi} (\log \log n)^{1/2}) + O(n^{-\xi/2}(\log n)^{1/2}) + O(n^{-\xi/2} (\log n)^{1/2} (\log \log n)^{1/2}) +  O(n^{(\xi-\gamma)/2} \log \log n) \\
      &=O(n^{-\gamma + 2\xi} (\log \log n)^{1/2}) +  O(n^{-\xi/2} (\log n)^{1/2} (\log \log n)^{1/2})
\end{align*}
almost surely, since, by construction, $\xi < \gamma/2$ which implies that $n^{(\xi-\gamma)/2} < n^{-\xi/2}$. Some thought finally reveals that the best rate is obtained by taking $\xi = 2\gamma/5$ which gives the rate in~\eqref{eq:stoc}.
\end{proof}

\bibliographystyle{myjmva}
\bibliography{biblio}

\end{document}